\documentclass[11pt,a4paper]{article}
\usepackage[lmargin=1.0in,rmargin=1.0in,bottom=1.0in,top=1.0in,twoside=False]{geometry}
\usepackage{mathtools}
\usepackage{hchang}

\let\qedhere\relax

\nolinenumbers

\usepackage{amsthm}
\usepackage{graphicx}%,algorithmic,algorithm}%, verbatim}
\usepackage{enumerate}
\usepackage{tikz}
\usepackage{stackengine}
\usepackage{accents}
\usepackage{textpos}
\usetikzlibrary{shapes,calc,fit}

\usepackage[T1]{fontenc}

\usepackage{enumitem}

\usepackage{microtype}
\usepackage{comment}
\usepackage[capitalize,nameinlink]{cleveref}

\crefname{figure}{Figure}{Figures}
\Crefname{figure}{Figure}{Figures}

\usepackage{stmaryrd}

\newtheorem{lemma}{Lemma}
\newtheorem{theorem}{Theorem}

\newtheorem{corollary}{Corollary}

\newtheorem{conjecture}{Conjecture}
\newtheorem{observation}{Observation}

\newcommand{\Oh}{\mathcal{O}}

\newcommand{\real}{\mathbb{R}}
\newcommand{\vis}{\mathcal V}

\newcommand{\hkb}{n_{\mathrm{HKB}}}
\DeclareMathOperator{\ex}{ex}

\author{\'Edouard Bonnet\\
        {CNRS, ENS de Lyon, Universit\'e Claude Bernard Lyon 1, LIP, UMR 5668}\\
        \href{mailto:edouard.bonnet@ens-lyon.fr}{edouard.bonnet@ens-lyon.fr}
}

\title{Large Finite Point Sets Have 4 Collinear Points or a 6-Clique}

\makeatletter
\AtBeginDocument{%
  \patchcmd{\@maketitle}
    {{\LARGE \@title \par}\vskip 1.5em}
    {{\LARGE \@title \par}\vskip 3.5em}
    {}
    {\PackageError{main}{Could not change title--author spacing}{}}
}
\makeatother

\begin{document}

\date{}

\maketitle

\begin{abstract}
  We prove that every finite point set of size at~least~$10^{11055931}$ has four collinear points or six points that pairwise see each other.
  This resolves the first open case of the big-line-big-clique conjecture of Kára, Pór, and Wood. 
\end{abstract}

\section{Introduction}

Given a~point set $P \subseteq \real^2$, two distinct points $p, p' \in P$ \emph{see} each other or are \emph{visible} (implicitly, with respect to~$P$) if the only points of~$P$ on the line segment $\overline{pp'}$ are $p$ and~$p'$.
If instead there is a~point $q \in P \setminus \{p,p'\}$ that lies on $\overline{pp'}$, we say that $q$ \emph{blocks} the pair $p, p'$.
The visibility graph of~$P$, denoted by $\vis(P)$, has vertex set $P$ and an edge between two distinct points of~$P$ whenever they see each other in~$P$.
We denote by $\chi$ and $\omega$ the chromatic number and the clique number, respectively.

Kára, Pór, and Wood conjectured that every sufficiently large finite planar point set has many collinear points or many pairwise visible points~\cite{Kara05}.
This is commonly referred to as the big-line-big-clique conjecture.

\begin{conjecture}[Big-line-big-clique~\cite{Kara05}]\label{conj:blbc}
  For all positive integers $k, \ell$, there is an integer $n := n(k, \ell)$ such that every finite point set $P \subseteq \real^2$ of size at~least~$n$ has $\ell$ collinear points or $\omega(\vis(P)) \geqslant k$.  
\end{conjecture}

The big-line-big-clique conjecture fails for infinite point sets: there is an infinite point set without four collinear points whose visibility graph is triangle-free~\cite{Por10inf}.
A~simple projection argument shows that the conjecture is equivalent if stated in $\mathbb R^d$ for any dimension $d \geqslant 2$~\cite{Kara05}. 

The conjecture is trivial for $\ell=3$ (thus $\ell \leqslant 3$) and every $k$ because if no three points are collinear, then $\vis(P)$ is a~complete graph.
The authors of~\cref{conj:blbc} showed the case $k \leqslant 4$ and every $\ell$~\cite{Kara05}.
The case $k = 5$ and $\ell=4$ was established \cite{Addario07}.
Then \cref{conj:blbc} was proved for $k = 5$ and every $\ell$~\cite{Abel11}, and later the bound~in~$\ell$ was improved to~$\Theta(\ell^2)$~\cite{Barat15}.
Prior to our work, the conjecture was open for all the other pairs $(k, \ell)$.

\medskip
In this note, we show the first open case: $k=6$ and $\ell=4$.
This marks the first resolution of an open case since 2009, when the arXiv version of~\cite{Abel11} appeared.

\begin{theorem}\label{thm:main}
Every finite point set $P$ of size at~least~$10^{11055931}$ contains 4~collinear points or $\omega(\vis(P)) \geqslant 6$. 
\end{theorem}

We rely on the following weakening of~\cref{thm:main}, due to Hujter and Kisfaludi-Bak.

\begin{theorem}[\cite{Hujter14}]\label{thm:hkb}
 Every finite point set $P$ of size at~least~2311 contains 4~collinear points or $\chi(\vis(P)) \geqslant 6$.
\end{theorem}

Visibility graphs of point sets are not $\chi$-bounded: they can have arbitrarily large chromatic number and clique number~6~\cite{Pfender08}.
This rules out a~general approach to~\cref{conj:blbc} solely based on showing its chromatic weakening (where $\omega$ is replaced by~$\chi$).

\paragraph*{Proof outline.}
Let $P \subseteq \real^2$ be any finite point set without four collinear points.
We set $G := \vis(P)$, and further assume that $\omega(G) < 6$.
We upper-bound $n := \abs{P}$ to prove the contrapositive of~\cref{thm:main}.

We first totally order~$P$: $p_1, p_2, \ldots, p_n \in P$ such that every interval $p_i, p_{i+1}, \ldots, p_j$ induces the same visibility graph as in~$G$.
(The order is the left-to-right order on an injective projection of $P$ onto a~line.) 
By the result of Hujter and Kisfaludi-Bak (\cref{thm:hkb}), any sufficiently long interval $I$ cannot be made 5-partite (i.e., of chromatic number at~most~5) by removing less than a~constant fraction of~$I$.
Otherwise, a~still large remaining subinterval of~$I$ contradicts~\cref{thm:hkb}.

As a~consequence of a~vertex-removal version of the Erd\H{o}s--Simonovits stability theorem (\cref{lem:turan-stab}), this implies that $G[I] = \vis(I)$ has at~most $(\frac{2}{5}-\varepsilon_1)\abs{I}^2$ edges for some constant~$\varepsilon_1 > 0$.
Passing to the complement $B := \overline{G}$, we get that $B[I]$ has at~least $c_{\ref{cor:lb-edges}}\abs{I}^2 - \Oh(\abs{I})$ edges for some constant $c_{\ref{cor:lb-edges}} > \frac{1}{10}$.
The exact value $\frac{1}{10}$ and the strict inequality will matter.

Next, we introduce the following potential, intending to show conflicting upper and lower bounds on it in terms of~$n$,
\[W := \sum\limits_{p_i p_j \in E(B), i < j} \frac{1}{j-i},\] summing the reciprocal of the positive \emph{stretch} for every nonvisible pair $p_i, p_j \in P$.
Observe that $\sum_{1 \leqslant i < j \leqslant n} \frac{1}{j-i} = n \ln n + \Oh(n)$.
Now, because $P$ has no four collinear points, we can show that every unit of potential from a~nonvisible pair $p_i, p_j$ forces four units of potential from \emph{visible} pairs; those five units are only used for $p_i, p_j$ and no other nonvisible pair.
Therefore, $W \leqslant \frac{1}{5}\,n \ln n + \Oh(n)$.

We now wish to leverage the lower bound on $\abs{E(B[I])}$ for every sufficiently large interval~$I$ to get a~lower bound on~$W$.
Partitioning $[n]$ into intervals of length~$m$, one gets $\left(\frac{n}{m}-\Oh(1)\right) \left(c_{\ref{cor:lb-edges}}m^2 - \Oh(m)\right) = c_{\ref{cor:lb-edges}}nm - \Oh(n+m^2)$ edges within these intervals.  
Note that \[A_m := \sum\limits_{p_ip_j \in E(B),~j-i \in [m-1]} \left(1-\frac{j-i}{m}\right)\]
is the expected number of edges of~$B$ within the intervals, thus in particular, $A_m \geqslant c_{\ref{cor:lb-edges}}nm - \Oh(n+m^2)$.
It can be shown that $\sum_{m = j-i+1}^n \frac{2}{m^2-1} \left(1-\frac{j-i}{m}\right) \leqslant \frac{1}{j-i}$.
By summing this inequality over every nonvisible pair $p_i, p_j$ and swapping the summation order, we get \[W \geqslant \sum_{m=2}^n \frac{2}{m^2-1}\,A_m.\]
With the preceding lower bound on~$A_m$, this gives $W \geqslant 2c_{\ref{cor:lb-edges}}n \sum_{m=2}^n \frac{1}{m} - \Oh(n) = 2c_{\ref{cor:lb-edges}}n \ln n - \Oh(n)$.
We conclude that $n$ cannot be arbitrarily large from the upper bound on~$W$ of the previous paragraph, since $2c_1 > \frac{1}{5}$.
The section organization follows the paragraphs of the outline.

\section{Intervals and lower bound on distance to 5-colorability} 

We use $[n]$ as a~shorthand for $\{1, \ldots, n\}$.
We use the standard graph-theoretic notation.
For any graph $H$, $V(H)$ and $E(H)$ denote the vertex set and edge set, respectively, of~$H$.
If $S \subseteq V(H)$, then $H[S]$ denotes the subgraph of~$H$ induced by~$S$, and $H-S := H[V(H) \setminus S]$.

Fix a~non-vertical line $L$ such that the orthogonal projection $\pi$ of~$P$ onto~$L$ gives $|P|$ distinct points.
Let $p_1, \ldots, p_n$ enumerate the points of~$P$ by the left-to-right order of their projections $\pi(p_1), \ldots, \pi(p_n)$.
We now call any set of consecutive points $\{p_i, p_{i+1}, \ldots, p_j\}$, with $i, j \in [n]$ and $i \leqslant j$, an \emph{interval} of~$P$.

\begin{observation}\label{obs:vis-interval}
  For every interval $I$ of~$P$, it holds that $\vis(I) = G[I]$.
\end{observation}

\begin{proof}
  Two points of $I$ that are visible in $P$ are a~fortiori visible in~$I$.
  Conversely, if two points $p_i, p_j$ of~$I$ are blocked by a~point $q \in P$, then the projection $\pi(q)$ lies on the line segment $\pi(p_i)\pi(p_j)$.
  Thus $q \in I$, and $p_i$ and $p_j$ do not see each other in~$I$ either.
\end{proof}

Let $\hkb := 2310$, where the subscript stands for the authors of~\cref{thm:hkb}.
The latter theorem can be rephrased as follows: every finite point set with no four collinear points and with a~5-colorable visibility graph has size at~most~$\hkb$. 
We set \[\varepsilon_0 := \frac{1}{2(\hkb+1)}.\] 
We show that, due to \cref{thm:hkb}, the visibility graph of every sufficiently large interval of~$P$ cannot be made 5-colorable by removing less than an $\varepsilon_0$ fraction of its vertices.  

\begin{lemma}\label{lem:distance-5-col}
  For every interval $I$ of~$P$ with $\abs{I} \geqslant 2 \hkb$, and every $Z \subseteq I$,
  \[\chi(G[I]-Z) \leqslant 5~~\text{implies that}~~\abs{Z} \geqslant \varepsilon_0 \abs{I}.\]
\end{lemma}

\begin{proof}
  Indeed, assume that $\chi(G[I]-Z) \leqslant 5$.
  Let $z := \abs{Z}$.
  The set $Z$ cuts the interval $I$ in at~most~$z+1$ subintervals.
  More precisely, $I \setminus Z$ is the disjoint union of at~most~$z+1$ intervals $I_1, \ldots, I_h$ of~$P$.

  For every $j \in [h]$, the interval $I_j$ has no four collinear points, by assumption on~$P$.
  Furthermore $\vis(I_j) = G[I_j]$ by~\cref{obs:vis-interval}.
  Therefore, \[\chi(\vis(I_j)) = \chi(G[I_j]) \leqslant \chi(G[I]-Z) \leqslant 5,\]
  where the first inequality holds since $G[I_j]$ is an induced subgraph of~$G[I]-Z$.
  (Note that it is irrelevant that the latter graph is not necessarily a~visibility graph.)

  The previous paragraph implies, by~\cref{thm:hkb}, that $\abs{I_j} \leqslant \hkb$.
  Consequently, \[\abs{I} - z \leqslant h \hkb \leqslant (z+1) \hkb.\]
  Thus, \[z \geqslant \frac{\abs{I} - \hkb}{\hkb+1} \geqslant \frac{\abs{I}}{2} \cdot \frac{1}{\hkb+1} = \varepsilon_0 \abs{I},\]
  where the second inequality holds because $\abs{I} \geqslant 2 \hkb$.
\end{proof}

\section{Turán stability and edge density}

For any graph $H$, we use $e(H)$ as a~shorthand for $\abs{E(H)}$, i.e., the number of edges of~$H$.
We denote by $\ex(m,K_6)$ the maximum number of edges of a~$K_6$-free (i.e., without a~6-vertex clique) $m$-vertex graph.  
We need a~vertex-removal version of the Erd\H{o}s--Simonovits stability theorem~\cite{Erdos66,Erdos67,Simonovits68}.
The proof of \cite[Lemma 2.3]{Popielarz18} gives the following explicit form for~$r=5$.

\begin{lemma}[{\cite[Lemma 2.3]{Popielarz18}} for $r=5$]\label{lem:turan-stab}
  For every $\varepsilon \in (0,1/3750)$, every $K_6$-free graph $H$ 
  \[\text{with}~m \geqslant 20~\text{vertices and more than}~\ex(m,K_6)-\varepsilon m^2~\text{edges}\]
  admits a~set $Z$ of fewer than $3500 \varepsilon m$ vertices such that $\chi(H-Z) \leqslant 5$. 
\end{lemma}

We get the following useful consequence from the previous two lemmas.

\begin{lemma}\label{lem:ub-edges}
  Set $\varepsilon_1 := \varepsilon_0/3500$.
  For every interval $I$ of~$P$ with $m := \abs{I} \geqslant 2 \hkb$,
  \[e(G[I]) \leqslant \ex(m,K_6) - \varepsilon_1 m^2.\]
\end{lemma}

\begin{proof}
  We apply \cref{lem:turan-stab} with $\varepsilon := \varepsilon_1 < 1/3750$ and $H := \vis(I) = G[I]$.
  The conclusion that there is some set $Z \subseteq V(H)$ such that $\abs{Z} < \varepsilon_0 m$ and $\chi(H-Z) \leqslant 5$ is ruled out by \cref{lem:distance-5-col}.
  Since $H$ is $K_6$-free (as an induced subgraph of~$G$) and $m = |V(H)| \geqslant 20$, we have $e(H) \leqslant \ex(m,K_6)-\varepsilon_1 m^2$.
\end{proof}

We now translate \cref{lem:ub-edges} into a~lower bound for $B := \overline G$, the complement of~$G$.

\begin{corollary}\label{cor:lb-edges}
  Set $c_{\ref{cor:lb-edges}} := \frac{1}{10} + \varepsilon_1$.
  For every interval $I$ of~$P$ with $m := \abs{I} \geqslant 2 \hkb$,
  \[e(B[I]) \geqslant c_{\ref{cor:lb-edges}} m^2 - \frac{m}{2}.\]
\end{corollary}

\begin{proof}
  By~\cref{lem:ub-edges}, \[e(G[I]) \leqslant \ex(m,K_6) - \varepsilon_1 m^2 \leqslant \left(1-\frac{1}{5}\right)\frac{m^2}{2} - \varepsilon_1 m^2 = \frac{2}{5}\,m^2 - \varepsilon_1 m^2,\]
  where the second inequality is by Tur\'an's theorem~\cite{Turan41}.
  Thus,
  \[e(B[I]) \geqslant \binom{m}{2} - \frac{2}{5}\,m^2 + \varepsilon_1 m^2 = \frac{m^2}{2} - \frac{2}{5}\,m^2 + \varepsilon_1 m^2 - \frac{m}{2} = \left(\frac{1}{10}+\varepsilon_1\right)m^2 - \frac{m}{2}.\qedhere\]
\end{proof}

The constant $\frac{1}{10}$ and the strict inequality $c_{\ref{cor:lb-edges}} > \frac{1}{10}$ will turn out to be crucial.

\section{Harmonic upper bound}

We define the weight of a~pair of points $p_i, p_j \in P$, with $i < j$, by $w(p_i,p_j) := \frac{1}{j-i}$.
Note that every weight is positive.
We introduce \[W := \sum\limits_{\substack{p_ip_j \in E(B)\\i<j}} w(p_i,p_j) = \sum\limits_{\substack{p_ip_j \in E(B)\\i<j}} \frac{1}{j-i}.\]
Thus, $W$ is the sum of weights of pairs of points that do not see each other in~$P$.
The rest of the proof consists of showing conflicting upper and lower bounds on~$W$, resulting in an upper bound on~$n = \abs{P}$.

In this section, we show the following upper bound on~$W$.
The proof hinges on the absence of four collinear points in~$P$.

\begin{lemma}\label{lem:ub-W}
  $W \leqslant \frac{1}{5}\,n \ln n + \frac{n}{5}$.
\end{lemma}

\begin{proof}
  Every nonvisible pair $p_i, p_j \in P$ with $i < j$ admits a~\emph{unique} blocker $p_k \in P$, for otherwise there would be four points of~$P$ on a~single line.
  We denote by $k(i,j)$ the index of this unique blocker.
  Furthermore, we have $i < k(i,j) < j$.
  We charge the weight of $(p_i, p_j)$ to the weights of the three pairs $(p_i, p_{k(i,j)})$, $(p_{k(i,j)}, p_j)$, and $(p_i, p_j)$.
  Note that these three pairs cannot be charged by another nonvisible pair: this would again imply the existence of four collinear points in~$P$.

  We set $a := k(i,j) - i$ and $b := j - k(i,j)$.
  Thus, $a+b = j - i$.
  Therefore, \[w(p_i,p_{k(i,j)})+w(p_{k(i,j)},p_j)+w(p_i,p_j) = \frac{1}{a} + \frac{1}{b} + \frac{1}{a+b}.\]
  Since $(a-b)^2 \geqslant 0$, we have $a^2+b^2 \geqslant 2ab$, so $\frac{a}{b} + \frac{b}{a} \geqslant 2$, as $a, b > 0$.
  This implies that \[\frac{a+b}{b} + \frac{a+b}{a} \geqslant 4,~~\text{thus}~~\frac{1}{a} + \frac{1}{b} \geqslant \frac{4}{a+b}.\]
  Therefore,
  \[w(p_i,p_{k(i,j)})+w(p_{k(i,j)},p_j)+w(p_i,p_j) \geqslant \frac{4}{a+b} + \frac{1}{a+b} = \frac{5}{a+b} = 5\,w(p_i,p_j).\]
  We conclude that
  \[W = \sum\limits_{\substack{p_ip_j \in E(B)\\i<j}} w(p_i,p_j) \leqslant \frac{1}{5} \sum\limits_{\substack{p_ip_j \in E(B)\\i<j}} \left(w(p_i,p_{k(i,j)})+w(p_{k(i,j)},p_j)+w(p_i,p_j)\right) \leqslant \frac{1}{5} \sum\limits_{1 \leqslant i < j \leqslant n} \frac{1}{j-i},\]
  where the second inequality holds because the weights $w(p_i,p_{k(i,j)}), w(p_{k(i,j)},p_j), w(p_i,p_j)$ are used only once, by the argument in the first paragraph.
  Finally,
  \[\sum\limits_{1 \leqslant i < j \leqslant n} \frac{1}{j-i} \leqslant n \sum\limits_{1 \leqslant i \leqslant n-1} \frac{1}{i} \leqslant n \ln n + n,~~\text{and}~~W \leqslant \frac{1}{5}\,n \ln n + \frac{n}{5}.\qedhere\]
\end{proof}

Again, the exact coefficient $\frac{1}{5}$ of $n \ln n$ (and the fact that it is twice $\frac{1}{10}$) will be crucial. 

\section{Harmonic lower bound}

We say that a~pair $p_i, p_j \in P$ has \emph{stretch}~$d$ if $\abs{i-j}=d$.
We first need the following preparatory lemma, which determines the coefficients $\lambda_m$ for which $\sum_{m > d} \lambda_m \left(1-\frac{d}{m}\right)$ is equal to $\frac{1}{d}$, i.e., to the weight of a~pair $p_i, p_j$ of stretch~$d$.

\begin{lemma}\label{lem:telescoping}
  For every positive integer $d$, \[\sum\limits_{m = d+1}^{+\infty} \frac{2}{m^2-1} \left(1-\frac{d}{m}\right) = \frac{1}{d}.\]
\end{lemma}

\begin{proof}
  First note that $\frac{2}{m^2-1} = \frac{1}{m-1} - \frac{1}{m+1}$.
  Thus,
  \[\sum\limits_{m = d+1}^{+\infty} \frac{2}{m^2-1} \left(1-\frac{d}{m}\right) = \sum\limits_{m = d+1}^{+\infty} \left(\frac{1}{m-1} - \frac{1}{m+1}\right) - d \left( \sum\limits_{m = d+1}^{+\infty} \frac{1}{m(m-1)} - \sum\limits_{m = d+1}^{+\infty} \frac{1}{m(m+1)} \right).\]
  By telescoping, 
  \[\sum\limits_{m = d+1}^{+\infty} \left(\frac{1}{m-1} - \frac{1}{m+1}\right) = \frac{1}{d} + \frac{1}{d+1}.\]
  Moreover, $\frac{1}{m(m-1)}=\frac{1}{m-1} - \frac{1}{m}$, so
  \[ \sum\limits_{m = d+1}^{+\infty} \frac{1}{m(m-1)} = \sum\limits_{m = d+1}^{+\infty} \left(\frac{1}{m-1} - \frac{1}{m}\right) = \frac{1}{d}~~\text{and}~~\sum\limits_{m = d+1}^{+\infty} \frac{1}{m(m+1)}= \frac{1}{d+1}.\]
  Therefore,
  \[\sum\limits_{m = d+1}^{+\infty} \frac{2}{m^2-1} \left(1-\frac{d}{m}\right) = \frac{1}{d} + \frac{1}{d+1} - \frac{d}{d} + \frac{d}{d+1} = \frac{1}{d}.\qedhere\]
\end{proof}

We can now show the following lower bound on~$W$ by leveraging \cref{cor:lb-edges}.
We set the constant $m_0 := 2 \hkb = 4620$.

\begin{lemma}\label{lem:lb-W}
  $W \geqslant 2 c_{\ref{cor:lb-edges}} n \ln n - 3c_{\ref{cor:lb-edges}} \ln(4m_0)\,n$.
\end{lemma}

\begin{proof}
  We have
  \[ W = \sum\limits_{\substack{p_ip_j \in E(B)\\i<j}} \frac{1}{j-i} = \sum\limits_{d \geqslant 1}~\sum\limits_{\substack{p_ip_j \in E(B)\\ j-i = d}}~\frac{1}{d}.\]
  By~\cref{lem:telescoping}, and since every term of the series is positive, we have
  \[\frac{1}{d} \geqslant \sum\limits_{m=d+1}^n \frac{2}{m^2-1} \left(1-\frac{d}{m}\right).\]
  Thus,
  \[ W \geqslant \sum\limits_{d \geqslant 1}~\sum\limits_{\substack{p_ip_j \in E(B)\\ j-i = d}}~\sum\limits_{m=d+1}^n \frac{2}{m^2-1} \left(1-\frac{d}{m}\right)
  = \sum\limits_{\substack{p_ip_j \in E(B)\\i<j}}~\sum\limits_{m=j-i+1}^n \frac{2}{m^2-1} \left(1-\frac{j-i}{m}\right).\]
  Interchanging the order of summation, we get
  \[ W \geqslant \sum\limits_{m=2}^n~\sum\limits_{\substack{p_ip_j \in E(B)\\ j-i \in [m-1]}} \frac{2}{m^2-1} \left(1-\frac{j-i}{m}\right)
  = \sum\limits_{m=2}^n \frac{2}{m^2-1} \sum\limits_{\substack{p_ip_j \in E(B)\\ j-i \in [m-1]}} \left(1-\frac{j-i}{m}\right).\]
  Thus
  \begin{equation}\label{eq:W}
    W \geqslant \sum\limits_{m=2}^n \frac{2}{m^2-1}\,A_m,
  \end{equation}
  by setting
  \[A_m := \sum\limits_{\substack{p_ip_j \in E(B)\\ j-i \in [m-1]}} \left(1-\frac{j-i}{m}\right).\]

  Our goal is now to lower-bound~$A_m$ in some regime of~$m$.

  Consider the $m$ ways of partitioning the integers into intervals of size~$m$, intersected with~$[n]$.
  We say that an interval partition \emph{captures} an edge~$e$ if both endpoints of~$e$ are in the same interval of the partition.
  Drawing one of the $m$ partitions uniformly at random, the probability that an edge of~$B$ of stretch~$d$ is captured is $\max(1-\frac{d}{m},0)$.
  Thus $A_m$ is the expected number of edges of~$B$ captured by the partition.

  Fix any $m \geqslant m_0 = 2 \hkb$ with $m \leqslant n' := \lfloor \frac{n}{4} \rfloor$.
  We assume that $m_0 \leqslant n'$, since otherwise the lemma statement trivially holds from the nonnegativity of~$W$.
  For each partition, there are at least $\frac{n}{m}-2$ intervals of size~$m$ fully contained in~$[n]$.
  Thus, by~\cref{cor:lb-edges}, \emph{any} partition captures at least
  \[\left(\frac{n}{m}-2\right) \left(c_{\ref{cor:lb-edges}}m^2-\frac{m}{2}\right).\]
  In particular,
  \[A_m \geqslant \left(\frac{n}{m}-2\right) \left(c_{\ref{cor:lb-edges}}m^2-\frac{m}{2}\right) \geqslant c_{\ref{cor:lb-edges}}mn - 2c_{\ref{cor:lb-edges}}m^2 - \frac{n}{2}.\]
  Together with \cref{eq:W}, this implies that
  \[ W \geqslant \sum\limits_{m=2}^n \frac{2}{m^2-1}\,A_m \geqslant \sum\limits_{m=m_0}^{n'} \frac{2}{m^2-1}\,A_m \geqslant \sum\limits_{m=m_0}^{n'} \frac{2}{m^2-1} \left(c_{\ref{cor:lb-edges}}mn - 2c_{\ref{cor:lb-edges}}m^2 - \frac{n}{2}\right).\]
  Therefore,
  \begin{equation}\label{eq:three-sums}
    W \geqslant 2c_{\ref{cor:lb-edges}}n \sum\limits_{m=m_0}^{n'} \frac{m}{m^2-1} - 4c_{\ref{cor:lb-edges}} \sum\limits_{m=m_0}^{n'} \frac{m^2}{m^2-1} - n \sum\limits_{m=m_0}^{n'} \frac{1}{m^2-1}.
  \end{equation}
  We observe that
  \[\sum\limits_{m=m_0}^{n'} \frac{m}{m^2-1} \geqslant \ln \frac{n'+1}{m_0} \geqslant \ln \frac{n}{4m_0}.\]
  Moreover,
  \[ \sum\limits_{m=m_0}^{n'} \frac{m^2}{m^2-1} \leqslant  2n' \leqslant \frac{n}{2}~~\text{and}~~\sum\limits_{m=m_0}^{n'} \frac{1}{m^2-1} \leqslant \frac{3}{4}.\]
  We therefore conclude from \cref{eq:three-sums} that
  \[ W \geqslant 2c_{\ref{cor:lb-edges}}n \ln n - \left(2c_{\ref{cor:lb-edges}} \ln(4m_0) + 2c_{\ref{cor:lb-edges}}  + \frac{3}{4}\right) n
  \geqslant 2c_{\ref{cor:lb-edges}}n \ln n - 3c_{\ref{cor:lb-edges}} \ln(4m_0)\,n. \qedhere\]
\end{proof}

We can now finish the proof of the main result.

\begin{proof}[Proof of~\cref{thm:main}]
  \Cref{lem:ub-W,lem:lb-W} imply that
  \[2c_{\ref{cor:lb-edges}}n \ln n- 3c_{\ref{cor:lb-edges}} \ln(4m_0)\,n \leqslant \frac{1}{5}\,n \ln n + \frac{n}{5}.\]
  Therefore,
  \[\left(2c_{\ref{cor:lb-edges}} - \frac{1}{5}\right) \ln n \leqslant 3c_{\ref{cor:lb-edges}} \ln(4m_0) + \frac{1}{5}.\]
  Since the constant $2c_{\ref{cor:lb-edges}} - \frac{1}{5}$ is positive (recall that $c_{\ref{cor:lb-edges}} > \frac{1}{10}$),
   \[\ln n \leqslant \frac{3c_{\ref{cor:lb-edges}} \ln(4m_0) + \frac{1}{5}}{2c_{\ref{cor:lb-edges}} - \frac{1}{5}},~\text{thus}~~n \leqslant \exp\left(\frac{3c_{\ref{cor:lb-edges}} \ln(4m_0)+\frac{1}{5}}{2c_{\ref{cor:lb-edges}} - \frac{1}{5}}\right) < 10^{11055931}. \qedhere\]
\end{proof}

\paragraph*{AI disclosure.}
After one fruitless attempt and the now customary generic encouragement, a~relatively detailed proof of~\cref{thm:main} was provided by GPT-5.6 Sol Pro after pondering for 222 minutes.
The author's contributions were limited to checking the proof, simplifying some parts, and eventually writing the paper in a~way that would give the author (and hopefully other human readers) a~more pleasant reading experience.  

\bibliographystyle{plain}
\bibliography{main}

@article{Erdos66,
  title={A limit theorem in graph theory},
  author={Erd\H{o}s, Paul and Simonovits, Mikl{\'o}s},
  journal={Studia Sci. Math. Hungar},
  volume={1},
  pages={51--57},
  year={1966}
}

@article{Erdos67,
  title={Some recent results on extremal problems in graph theory ({Results})},
  author={Erd\H{o}s, Paul},
  journal={Theory of Graphs (Internat. Sympos., Rome, 1966)},
  pages={117--123},
  year={1967}
}

@inproceedings{Simonovits68,
  title={A method for solving extremal problems in graph theory, stability problems},
  author={Simonovits, Mikl{\'o}s},
  booktitle={Theory of Graphs (Proc. Colloq., Tihany, 1966)},
  pages={279--319},
  year={1968}
}

@article{Hujter14,
  title={5 Colorable Visibility Graphs Have Bounded Size or 4 Collinear Points},
  author={Hujter, B{\'a}lint and Kisfaludi-Bak, S{\'a}ndor},
  journal={arXiv preprint arXiv:1410.7273},
  year={2014}
}

@article{Turan41,
  title={Eine {Extremalaufgabe} aus der {Graphentheorie}},
  author={Tur{\'a}n, Paul},
  journal={Mat. Fiz. Lapok},
  volume={48},
  pages={436--452},
  year={1941}
}

@article{Kara05,
  author       = {Jan K{\'{a}}ra and
                  Attila P{\'{o}}r and
                  David R. Wood},
  title        = {On the Chromatic Number of the Visibility Graph of a Set of Points
                  in the Plane},
  journal      = {Discret. Comput. Geom.},
  volume       = {34},
  number       = {3},
  pages        = {497--506},
  year         = {2005},
  url          = {https://doi.org/10.1007/s00454-005-1177-z},
  doi          = {10.1007/S00454-005-1177-Z},
  bibsource    = {dblp computer science bibliography, https://dblp.org}
}

@misc{Por10inf,
      title={The Big-Line-Big-Clique Conjecture is False for Infinite Point Sets}, 
      author={Attila~Pór and David R. Wood},
      year={2010},
      eprint={1008.2988},
      archivePrefix={arXiv},
      primaryClass={math.CO},
      note = {arXiv:1008.2988},
      url={https://arxiv.org/abs/1008.2988}, 
}

@article{Popielarz18,
author = {Kamil Popielarz and Julian Sahasrabudhe and Richard Snyder},
title = {A stability theorem for maximal {$K_{r+1}$}-free graphs},
journal = {Journal of Combinatorial Theory, Series B},
volume = {132},
pages = {236--257},
year = {2018},
issn = {0095-8956},
doi = {10.1016/j.jctb.2018.04.001},
url = {https://www.sciencedirect.com/science/article/pii/S0095895618300248}
}

@article{Abel11,
  title={Every large point set contains many collinear points or an empty pentagon},
  author={Abel, Zachary and Ballinger, Brad and Bose, Prosenjit and Collette, S{\'e}bastien and Dujmovi{\'c}, Vida and Hurtado, Ferran and Kominers, Scott Duke and Langerman, Stefan and P{\'o}r, Attila and Wood, David R.},
  journal={Graphs and {C}ombinatorics},
  volume={27},
  number={1},
  pages={47--60},
  year={2011},
  publisher={Springer}
}

@misc{Addario07,
  title={On a geometric {R}amsey-style problem},
  year={2007},
  author={Addario-Berry, Louigi and Fernandes, Cristina and Kohayakawa, Yoshiharu and de Pina, Jos Coelho and Wakabayashi, Yoshiko}
}

@article{Barat15,
  author       = {J{\'{a}}nos Bar{\'{a}}t and
                  Vida Dujmovi{\'c} and
                  Gwena{\"{e}}l Joret and
                  Michael S. Payne and
                  Ludmila Scharf and
                  Daria Schymura and
                  Pavel Valtr and
                  David R. Wood},
  title        = {Empty Pentagons in Point Sets with Collinearities},
  journal      = {{SIAM} J. Discret. Math.},
  volume       = {29},
  number       = {1},
  pages        = {198--209},
  year         = {2015},
  url          = {https://doi.org/10.1137/130950422},
  doi          = {10.1137/130950422},
  bibsource    = {dblp computer science bibliography, https://dblp.org}
}

@article{Pfender08,
  author       = {Florian Pfender},
  title        = {Visibility Graphs of Point Sets in the Plane},
  journal      = {Discret. Comput. Geom.},
  volume       = {39},
  number       = {1--3},
  pages        = {455--459},
  year         = {2008},
  url          = {https://doi.org/10.1007/s00454-008-9056-z},
  doi          = {10.1007/S00454-008-9056-Z},
  bibsource    = {dblp computer science bibliography, https://dblp.org}
}

\end{document}